\documentclass[12pt]{article}
\usepackage{amssymb,amsmath, amsthm, amscd,ifthen}
\usepackage[dvips]{graphicx}
\usepackage{indentfirst}

\newcommand{\Nat}{{\mathbb N}}
\newcommand{\Real}{{\mathbb R}}
\newcommand{\R}{{\mathbb R}}

\newcommand{\eps}{\varepsilon}
\renewcommand{\ss}{\frac{s^2}{6}}

\renewcommand{\le}{\leqslant}
\renewcommand{\ge}{\geqslant}

\renewcommand{\leq}{\leqslant}
\renewcommand{\geq}{\geqslant}

\newtheorem{theorem}{Theorem}
\newtheorem{claim}{Claim}
\newtheorem{prop}{Proposition}

\newtheorem{corollary}{Corollary}

\newtheorem{defin}{Definition}

\DeclareMathOperator{\aff}{aff}

\begin{document}

\title{New bounds for the distance Ramsey number\footnote{This work is supported by the grant N
12-01-00683 of the Russian Foundation for Basic Research and by the grant NSh-2519.2012.1 of the Leading Scientific Schools
of Russia.}}

\author{A.B. Kupavskii\footnote{Moscow State University, Mechanics and Mathematics Faculty, Department of Number Theory;
Moscow Institute of Physics and Technology, Faculty of Innovations and High Technology, Department of Discrete
Mathematics; Yandex research laboratories.}, A.M. Raigorodskii\footnote{Moscow State University, Mechanics and Mathematics Faculty, Department of Mathematical Statistics
and Random Processes; Moscow Institute of Physics and Technology, Faculty of Innovations and High Technology, Department of Discrete
Mathematics; Yandex research laboratories.},
M.V. Titova\footnote{Moscow State University, Mechanics and Mathematics Faculty, Department of Mathematical Statistics
and Random Processes.}}
\date{}
\maketitle

\begin{abstract} In this paper we study the distance Ramsey number $R_{{\it D}}(s,t,d)$. The \textit{distance Ramsey number} $R_{{\it D}}(s,t,d) $ is the minimum
number $n$ such that for any graph $ G $ on $ n $ vertices, either $G$ contains an induced $ s $-vertex
subgraph isomorphic to a distance graph in $ \Real^d $ or $ \bar {G} $ contains an induced $ t $-vertex subgraph isomorphic to the distance graph in $ \Real^d $. We obtain the upper and lower bounds on $R_{{\it D}}(s,s,d),$ which are similar to the bounds for the classical Ramsey number $R\left(\left \lceil \frac{s}{[d/2]} \right \rceil, \left \lceil \frac{s}{[d/2]} \right \rceil\right)$.

\end{abstract}

\begin{section} {Introduction}

In this paper we analyze properties of distance graphs from the point of view of Ramsey theory (see \cite{TRam}, \cite{Ram}).
Let us remind the notion of distance graph.

\begin{defin}
{\rm A graph $G$ is the} (unit) distance graph {\rm in $d$-dimensional Euclidean space $\R^d$ if} 
$$
V(G) \subseteq \R^d; \,\,\,\,\,\  E(G) \subseteq \{(x; y) \in V^2: \,\,\, |x-y| = 1\}.
$$
\end{defin}

The study of various properties of finite distance graphs was motivated by Erd\H{o}s' work \cite{Erd1}, where he stated three fundamental problems of combinatorial geometry. One of the problems is the following: how many can there be unit distances among $n$ points on the plane? In terms of distance graphs this question can be stated as follows. Let $G$ be a distance graph in $\R^2$. What is the maximum value of $|E(G)|$ provided that $|V(G)|=n$?

Another problem that is closely related to properties of distance graphs is the famous Nelson--Hadwiger problem on finding the chromatic number $ \chi(\Real^d) $ of the space (see \cite{Rai1}). On the one hand, for every distance graph $G$ in $\R^d$ we have $\chi(G)\le \chi(\R^d)$, where $ \chi(G) $ is the usual chromatic number of the graph. On the other hand, Erd\H{o}s-- de Bruijn theorem (see \cite{ErBr}) states that $ \chi\left(\Real^d\right) = \chi(H)$ for some \textit{finite }distance graph $H$ in $\R^d$.

These and other well-known problems such as Borsuk's partition problem (see \cite{Rai2}, \cite{Rai3})
give the motivation to analyze different properties of finite distance graphs (various problems concerning distance graphs can be found in \cite{BMP}).

Another combinatorial field, which lies at the basis of this work, is Ramsey theory. Recall the definition of the Ramsey numbers $R(s,t)$.

\begin{defin} {\rm Given} $ s,t \in \Nat,$ the \textit{Ramsey number $R(s,t)$} {\rm is the minimum
number $n$ such that for any graph $ G $ on $ n $ vertices,
either $G$ contains an $ s $-vertex independent set (i.e., a set without edges) or its complement $ \bar{G} $ contains a $t$-vertex independent set.}
\end{defin}

The main concept in this work is that of \textit{distance Ramsey number}.

\begin{defin} The \textit{distance Ramsey number} $R_{{\it D}}(s,t,d) $ {\rm is the minimum
number $n$ such that for any graph $ G $ on $ n $ vertices, either $G$ contains an induced $ s $-vertex
subgraph isomorphic to the distance graph in $ \Real^d $ or $ \bar {G} $ contains an induced $ t $-vertex subgraph isomorphic to the distance graph in $ \Real^d $.}
\end{defin}

Since for every $d\ge 1$ an independent set of any finite size can be realized as the distance graph in $\R^d$, we have the following obvious inequality: $ R_{{\it D}}(s,t,d) \le R(s,t) $.

Best known bounds for classical Ramsey numbers are the following:
$$
\frac{\sqrt{2}}{e} (1+o(1)) s 2^{\frac{s}{2}} \le R(s,s) \le
e^{-\gamma \frac{{\rm ln}^2\,s}{{\rm ln}\,{\rm ln}\,s}}
\cdot 4^s, \,\,\, \gamma > 0.
$$

The lower bound is due to Spencer and can be found in \cite{AS}, the upper bound is due to Conlon \cite{Con}.

Conlon's bound immediately implies the following upper bound on diagonal distance Ramsey numbers:

$$
R_{{\it D}}(s,s,d) \le 4^s\,e^{- \gamma\frac{\ln^2 s}{\ln \ln s}}, ~~ \gamma > 0.
$$

The concept of distance Ramsey number was introduced and studied in the paper\cite{Rai03}, in which several asymptotic lower bounds were obtained.
Distance Ramsey number was also studied in \cite{RKT} and \cite{RaiTi1}. In these papers authors introduced different methods to obtain lower bounds on $R_{{\it D}}(s,t,d)$ for the case of small fixed $d$. The sharpest bounds for $d\in\{2,\ldots,8\}$ are stated in the following theorems
(Theorems \ref{T:oldd2}, \ref{T:oldd3} see in \cite{RaiTi1},
Theorem \ref{T:oldd4d8} see in \cite{RKT}).

\begin{theorem} \label{T:oldd2} Let $d = 2$. There exists a positive constant $c$ such that

$$
R_{{\it D}}(s,s,d) \ge 2^{{\frac s2} - c\,s^{\frac 13}\ln s}.
$$

\end{theorem}

\begin{theorem} \label{T:oldd3} Let $d = 3$. There exists a positive constant $c$ such that

$$
R_{{\it D}}(s,s,d) \ge 2^{{\frac s2} - c\, \beta(s) s^{\frac 12}\ln s},
$$
where $ \beta(s) = 2^{\alpha^2(s)} $, and $ \alpha(s) $ is inverse Ackermann function.

\end{theorem}

\begin{theorem}\label{T:oldd4d8} Let $ d \in \{4, \dots, 8\} $. We have

$$
R_{{\it D}}(s,s,d) \ge \frac {1}{e \cdot 2^{ \frac{2^{d-1}-1}{2^d} } }(1+o(1)) k 2^{\frac k2},\qquad \hbox{where}\,~ k = \left[c_d s\right] \text{ and}$$

$$
c_4 = 0.04413, \,\, c_5 = 0.01833, \,\, c_6 = 0.00806, \,\, c_7 = 0.00352, \,\, c_8 = 0.00165.
$$

\end{theorem}

Proofs of these theorems rely on some special properties of distance graphs in small dimensions. In cases  $d=2, 3$ the sharpest bound is based on the fact that the number of edges in a distance graph on $n$ vertices in $\R^2,\ \R^3$  does not exceed $n^{2-\varepsilon}$ for some $\varepsilon >0.$
However, distance graphs do not have this property in spaces ${\mathbb{R}}^d$, $d=4,\ldots,8$. For every $m\in \mathbb{N}$ we can realize a complete bipartite graph $K_{m,m}$ as the distance graph in $\R^4$. Indeed, consider two circles 

$$C_1=\{(x_1,x_2,0,0)\in \R^4:x_1^2+x_2^2=1/2\}$$ and 
$$C_2=\{(0,0,x_3,x_4)\in \R^4:x_3^2+x_4^2=1/2\}.$$ Then, by Pithagoras' theorem, the distance between any point of $C_1$ and any point of $C_2$ equals 1. Hence, we can embed one part of $K_{m,m}$ into $C_1$, and the second part into $C_2$. In cases $d=4,\ldots,8$ the proofs of the bounds are based on the following type of claims: every $n$-vertex distance graph in $\mathbb{R}^d$ contains several non-overlapping independent sets of sufficiently large (depending on $n$) total cardinality.

In this paper we describe a method that allows us to obtain much sharper bounds on distance Ramsey number $R_{{\it D}}(s,s,d)$ for every fixed $ d \ge 4 $. We state the bounds in the following theorem and in proposition \ref{upper}.

\begin{theorem}\label{T:main}

Let $ d \ge 4. $ 
The following inequality holds:

$$
R_{{\it D}}(s,s,d) \geq 2^{\left(\frac{1}{2[d/2]}-\bar o(1)\right)s}.
$$

\end{theorem}

Theorem \ref{T:main} significantly strengthens the bounds from Theorem \ref{T:oldd4d8}. Moreover, Theorem \ref{T:main} gives essentially the same bounds
for $ d \in\{ 2,3\} $ as Theorems \ref{T:oldd2} and \ref{T:oldd3} do, though these theorems provide an explicit formula for the $\bar o(1)$
factor in the exponent. As we will see from the proof in general it is difficult to express this factor explicitly using the new method.

For a graph $G$ let $Cl(G,r)$ denote the number of $r$-cliques in $G,$ and put $cl(G,r)=|Cl(G,r)|$. To prove Theorem \ref{T:main} we need the following theorem.

\begin{theorem}\label{T:klika} For any fixed natural $d$ there exists $\eps > 0 $ and there exists $ n_0 \in {\mathbb N} $ such
that for every distance graph $G$ in $\R^d$ with $n \ge n_0$ vertices
$$
cl\left(G,{\left[\frac{d}{2}\right]+1}\right) \leq n^{\left[\frac{d}{2}\right]+1-\eps}.
$$
\end{theorem}

This theorem allows us to generalize the method used to obtain bounds in Theorems \ref{T:oldd2} and \ref{T:oldd3}.
We prove this theorem in Section \ref{sec2}. In Section 3 we present the proof of Theorem 4. Finally, in Section 4 we prove

\begin{prop}\label{upper} For any $1\le d\le s$ we have
$$
R_{{\it D}}(s,s,d) \leq 2 \left[\frac{d}{2}\right] R\left(\left \lceil \frac{s}{[d/2]} \right \rceil, \left \lceil \frac{s}{[d/2]} \right \rceil\right) \le
4^{\frac{s}{[d/2]}(1+o(1))}.
$$

\end{prop}

The proposition significantly strengthens the described above  trivial upper bound.
Moreover, the estimate for $R_{D}(s,s,d),$ which is given in Theorem \ref{T:main} and Proposition \ref{upper} turns out to be
essentially the same as for the classical Ramsey number $R\left(\left \lceil \frac{s}{[d/2]} \right \rceil, \left \lceil \frac{s}{[d/2]} \right \rceil\right)$:

\begin{equation*}
\frac {s}{2[d/2]}(1+\bar o(1))\le \log R_{{\it D}}(s,s,d) \le \frac {2s}{[d/2]}(1+\bar o(1)).
\end{equation*}

Therefore, in some sense we solve the problem completely for fixed $d$.

\end{section}

\begin{section}{Proof of Theorem \ref{T:klika}}\label{sec2}\end{section}

We use $K_{l_1,\dots,l_r}$ to denote a complete $r$-partite graph which parts have cardinalities $l_1$,$\dots$,$l_r$.

Theorem \ref{T:klika} follows from Proposition \ref{C:K333} and Corollary \ref{Cor:Erd} of Theorem \ref{T:Erd}. Let us begin with the proposition.

\begin{prop}\label{C:K333}

 If $G$ is a distance graph in $\Real^d$, then $G$ does not contain a subgraph isomorphic to $K_{\underbrace{3,\dots,3}_{\left[\frac{d}{2}\right]+1}}.$

\end{prop}

\begin{proof}
The proof uses induction on $ d $.

 First, we verify the proposition for $d\in\{2,3\}$. Suppose that the distance graph $G$ in $\Real^3$ has a subgraph, isomorphic to $K_{3,3}$. Consider three vertices $v_1$, $v_2$, $v_3$ from the first part.
The other vertices of the subgraph lie on the line $l$, that is orthogonal to plane $\aff \langle v_1,v_2,v_3 \rangle $ and passes through a circumcenter of the triangle with vertices $v_1$, $v_2$, $v_3$. But  the  line $l$ contains at most two points that lie at unit distance apart from $v_1$, $v_2$, $v_3$. Thus, the statement is true for $d\in\{2,3\}$.

Assume that the proposition holds for $d$. Consider a distance graph $G
\subset \R^{d+2}$. Suppose that it has a subgraph isomorphic to $K_{3,\dots,3}$ with $\left[\frac{d}{2}\right]+2$ parts. Again consider vertices $v_1$, $v_2$, $v_3$
from the first part. All other vertices of the subgraph lie in the hyperplane that is orthogonal to plane $\aff \langle v_1,v_2,v_3 \rangle $ and
passes through a circumcenter of the triangle  $v_1v_2v_3$. However, by the induction hypothesis there are no subgraphs in $d$-dimensional
space isomorphic to $K_{3,\dots,3}$ with $\left[\frac{d}{2}\right]+1$ parts. This contradiction concludes the proof.
\end{proof}

Next we state Theorem 6, which is proven in \cite{Erd}. We introduce some notation from \cite{Erd}.
Let $K^{(r)}(l_1,\dots,l_r)$ be a complete
$r$-partite $r$-uniform hypergraph which parts have cardinalities $l_1$, \dots, $l_r$ (every edge has exactly one vertex from every partite set), and let
$f\left(n; K^{(r)}(l_1,\dots,l_r)\right)$ be the least natural number such that every $ r $-uniform hypergraph with $ n $ vertices and
$$f\left(n; K^{(r)}(l_1,\dots,l_r)\right)$$ edges has a subhypergraph isomorphic to $K^{(r)}(l_1,\dots,l_r)$.

\begin{theorem}\label{therd} (Erd\H{o}s, \cite[Theorem 1]{Erd}.\label{T:Erd})
Let $n > n_0(r, l)$, $l > 1$. For sufficiently large $C$ ($C$ does not depend on $n,r,l$) the following inequality holds:

 $$f\left(n;K^{(r)}(l,\dots,l)\right) \leq n^{r-\dfrac{1}{l^{r-1}}}. $$

\end{theorem}

\begin{corollary} \label{Cor:Erd} For given $ l $ and $ r $ there exists $\eps > 0$ and $ n_0 \in {\mathbb N} $
such that if $ n \ge n_0 $ and $n$-vertex graph $G$ does not have a subgraph isomorphic to
$K_{\underbrace{l,\dots,l}_{r}},$ then $$cl(G,r) \leq n^{r-\eps}.$$
\end{corollary}

\begin{proof}

Indeed, consider a graph $G$ that does not contain a subgraph isomorphic to
$K_{\underbrace{l,\dots,l}_{r}}$. Construct a hypergraph $\widetilde{G}=(V,\widetilde{E})$ with the vertex set that is the
same as the vertex set of $G$ and with the edge set consisting of all the $r$-cliques of the graph $G$. Let  $|\widetilde{E}|=m$ and  suppose $m \geq f(n;K^{(r)}(l,\dots,l))$. Note that $m=cl(G,r)$. According to the definition, hypergraph $\widetilde{G}$ has a subhypergraph isomorphic to $K^{(r)}(l,\dots,l)$. Thus $G$ has a subgraph isomorphic to $K_{\underbrace{l,\dots,l}_{r}}$, which contradicts the assumption.

Hence $m < f(n;K^{(r)}(l,\dots,l))$.
By Theorem \ref{T:Erd} there exits $ \eps > 0$, $\eps=\eps(l,r)$, such that $m < n^{r-\eps}$.

\end{proof}

\begin{proof} [Proof of Theorem \ref{T:klika}]
Let $G$ be a distance graph in $\R^d$. By Proposition \ref{C:K333} $G$ does not contain $K_{\underbrace{3,\dots,3}_{\left[\frac{d}{2}\right]+1}}$. We apply
Corollary \ref{Cor:Erd}  with $r=[d/2]+1$ and $l=3$  to $G$ and get the statement of Theorem \ref{T:klika}.
\end{proof}

\begin{section}{Proof of Theorem \ref{T:main}}\label{sec4}\end{section}
\subsection{How to obtain lower bounds on $R_{{\it D}}(s,s,d)$}\label{s3.1}

To obtain a lower bound $R_{{\it D}}(s,s,d)>n$ for the distance Ramsey number  we need to prove that there exists such a graph $G$ on $n$ vertices that every induced $s$-vertex subgraph of $G$ and every induced $s$-vertex subgraph of $\bar{G}$
is not isomorphic to a distance graph in $\mathbb{R}^d$.

Let $k=[d/2]+1$, and let $ \eps = \eps(d) $ be the number from Theorem 5. Theorem 5 states that every graph $ H $ in $ \R^d $ on  $ s $ vertices has at most $ s^{k-\eps} $ $ k $-cliques. We will prove that for a specific natural $n$ there exists an $n$-vertex graph $G$ such that every induced $s$-vertex subgraph of $G$ and every induced $s$-vertex subgraph of its complement $\bar{G}$ contains more than $ s^{k-\varepsilon}$ cliques of size $k$.  In this case the inequality $R_{{\it D}}(s,s,d)>n$ takes place. The value $s$ is supposed to be sufficiently large (see Theorem 5 and Theorem 4).

We use probabilistic method (see, e.g., \cite{AS}).
For every natural $n$ consider the classical Erd\H{o}s -- R\'enyi random graph model $G\left(n,1/2\right)$ (see, e.g., \cite{AS}, \cite{Boll}).

For every subset $S$, $ |S| = s $, of the vertex set $V_n$ of a random graph $G \sim G\left(n,1/2\right)$ we define the event $A_S$: the graph $G[S]$ has at most $s^{k-\varepsilon}$ cliques of size $k$. We use $A'_S$ to denote the event that the graph
$\bar{G}[S]$ has at most $s^{k-\varepsilon}$ cliques of size $k$.

If we prove that for a certain $n$ there is a positive probability that none of the events  $ A_S, A_S' $ occur, i.e.
$$
P\left(\overline{\bigcup_{S \subset V_n} (A_S \cup A'_S)}\right) > 0,
$$
then we obtain the bound $R_{{\it D}}(s,s,d)>n$.

Fix positive $\gamma$. In the case of Theorem 4 we choose $ n$ equal to $2^{\left(\frac{1}{2[d/2]}-\gamma\right)s} $.
We prove that for any positive $\gamma$ the above described probability is positive, which, in turn, gives us the statement of the theorem.
To make the proof more transparent we begin with the case $d \in \{4, 5\}$. In these two cases we want to  bound the distance Ramsey number
by $ 2^{\left(\frac{1}{4}-\gamma\right)s} $ from below.

In Section \ref{s3.2} we deal with the case $d \in \{4, 5\}.$  The crucial part of the proof is to bound the probability of each event $ A_S, A_S' $, $ S \subset V_n$. First
we prove a weaker bound on the probability of single events, which is formulated in Theorem \ref{T:K3-K3}. It implies a weaker bound on the distance Ramsey number than the one we are to prove. Next we improve this bound using additional considerations, completing the proof of Theorem \ref{T:main} for $ d\in \{4, 5\} $. In Section \ref{s3.3}
we discuss the proof of Theorem \ref{T:main} for $ d\ge 6 $.
This sequence of presentation is intended to clarify the  method we use.

\bigskip

\subsection{Case $d\in \{4,5\}$}\label{s3.2}

In this case we have $k=3,$ so we deal with  triangles.

To bound the probability of each event $ A_S, A_S'$ accurately enough we need to prove several propositions.
For the sake of simplicity of presentation below we present a simpler method that doesn't give the sharpest bound.
Next we shortly describe how to modify it to obtain a better result.

\begin{theorem}\label{T:K3-K3} The following inequalities hold:
$$P(A_S)\le\mathcal{P}, \, P(A_S')\le\mathcal{P}, \hbox{where} \,\, \mathcal{P} =s! \cdot \left(\frac{7}{8}\right)^{{\frac{s^{2}}{6}}(1+o(1))}.$$
\end{theorem}

We will give the proof of Theorem \ref{T:K3-K3} below. First we state a corollary.

\begin{corollary}\label{Cor:K3-K3} For $d\in\{4,5\}$ we have the following lower bound for distance Ramsey number:

$$R_{{\it D}}(s,s,d) \geq \left({\frac 87} \right)^{{\frac s6} (1+o(1))} \approx 2^{0.032107s}.$$

\end{corollary}

\begin{proof}[Proof of corollary \ref{Cor:K3-K3}.]

We bound the probability of the union of the events $A_S,A'_S$ by the sum of probabilities:

$$
P\left(\bigcup_{S \subset V_n} (A_S \cup A'_S)\right) \leq \sum_{S \subset V_n}
(P(A_S)+P(A'_S)) \leq {n\choose s} \cdot s! \cdot \left(\frac{7}{8}\right)^{{\frac{s^{2}}{6}}(1+o(1))} \leq n^s
\cdot \left(\frac{7}{8}\right)^{{\frac{s^{2}}{6}}(1+o(1))}.
$$

Therefore, there exists a function $\alpha(s) = 1+o(1)$ such that if
$$
n \le \left({\frac 87} \right)^{{\frac s6}\alpha(s)},
$$
then the following inequality holds:
$$
P\left(\overline{\bigcup_{S \subset V_n} (A_S \cup A'_S)}\right)>0.
$$
\end{proof}

For the sake of brevity we use the notation  $T(G)$ instead of $Cl(G,3)$ and  $t(G)$ instead of $|T(G)|$.
To prove Theorem \ref{T:K3-K3} we need the well-known R\"odl's theorem (see \cite{Rodl}).

\begin{theorem}\label{Riddle}

Let $M$ denote a collection of $l$-sets of $\{1,\dots,n\}$ such that for all $A,B \in M$ holds $|A \cap B| \leq m-1$.
Put $g(l,m,n)=\max |M|$. For fixed $ l,m $ and for $ n \to \infty $ holds $ g(l,m,n) \sim \frac{{n\choose m}}{{l\choose m}}\ \left(\lim_{n\to \infty}\frac{g(l,m,n)}{{n\choose m}/{l\choose m}}=1\right).$
\end{theorem}

From now on we say that two graphs are {\it disjoint} if they have no edges in common.
Fix an arbitrary maximum system of pairwise disjoint triangles in the set $S=\{1,\dots,s\}$. We use $Tr(S)$ to denote this system.

\begin{corollary} (from Theorem \ref{Riddle})  Let $s\to\infty$. There exists $\psi(s)$, $\psi(s)\to 0$ as $s \to \infty,$ such that the following equality holds:
$$|Tr(S)| = \frac{s^2}{6}(1+\psi(s)).$$
\end{corollary}

Consider a graph $H = (S,E)$ of order $s$ and a permutation $\sigma$ of its vertex set $S$. Let $\sigma(H)$ denote the graph with edges
$\sigma(E) = \{(\sigma(a),\sigma(b)) \mid (a,b) \in E\}$. Consider the value $F(\sigma,H) = |T(\sigma(H))\cap Tr(S)|,$ which is the
number of triangles that the sets $T(\sigma(H))$ and $Tr(S)$ have in common.

We choose a random permutation (from the uniform distribution over all permutations) and find the expectation of $F(\sigma,H)$. Define the function $ \psi_1 $ from the following equation:
$$
\frac{s}{(s-1)(s-2)}(1+\psi(s)) = \frac{1}{s}(1+\psi_1(s)).
$$
It is clear that $ \psi_1(s) \to 0 $ as $ s \to \infty $.

\begin{claim}\label{C:Mat(F)}  For every graph $ H $ on $ s $ vertices the following holds:

 $$ \mathbb{E} (F(\sigma,H)) = \dfrac{|T(H)|}{s}(1+\psi_1(s)).$$

\end{claim}

\begin{proof}

We have:  $$ \mathbb{E} (F(\sigma,H)) = \sum_{\sigma}  (|T(\sigma(H))\cap Tr(S)|) \cdot \mathbb{P}(\sigma). $$ The number of common triangles can be calculated as follows.
Take a triangle $\Delta \in T(H)$. Consider the indicator function of the triangle $\sigma(\Delta)$ being an element of the set $Tr(S)$:

$$  \mathbb{I}(\sigma(\Delta) \in Tr(S)) =
\begin{cases}
1, &\text{ if } ~ \sigma(\Delta) \in Tr(S), \\
0, &\text{ if } ~ \sigma(\Delta) \not\in Tr(S).\\
\end{cases}
$$

We have $$|T(\sigma(H))\cap Tr(S)| = \sum_{\Delta \in T(H)} \mathbb{I}(\sigma(\Delta) \in Tr(S)).$$
Substituting this expression in the formula for the expectation of the number of common triangles we get

$$\sum_{\sigma}  (|T(\sigma(H))\cap Tr(S)|)\cdot \mathbb{P}(\sigma)= \sum_{\sigma} \sum_{\Delta \in T(H)} \mathbb{I}(\sigma(\Delta) \in Tr(S))
\cdot \mathbb{P}(\sigma) =$$
$$= \sum_{\Delta\in T(H)} \sum_{\sigma} \mathbb{I}(\sigma(\Delta) \in Tr(S)) \cdot \mathbb{P}(\sigma).$$

For every pair of triangles $\Delta, \Delta' \in Tr(S)$ the number of permutations $\sigma,$ such that $\sigma(\Delta)= \Delta',$ equals $(s-3)! \cdot 3!$ (there are $3!$ ways to rearrange vertices of the triangle $\Delta'$, the other vertices are permuted arbitrarily). Thus the number of permutations $\sigma$ such that $\sigma(\Delta) \in Tr(S),$ is equal to $(s-3)! \cdot 3! \cdot |Tr(S)|$.

Since $|Tr(S)| = \frac{s^2}{6}(1+\psi(s)),$ we have the following chain of equalities:
    
$$\sum_{\sigma} \mathbb{I}(\sigma(\Delta) \in Tr(S)) \cdot \mathbb{P}(\sigma) = \dfrac{(s-3)! \cdot 6 \cdot \frac{s^2}{6}}{s!}(1+\psi(s)) =
$$
$$
= \dfrac{s^2}{s(s-1)(s-2)}(1+\psi(s)) = \frac{1}{s} (1+\psi_1(s)).$$

This implies

$$\sum_{\Delta\in T(H)} \sum_{\sigma} \mathbb{I}(\sigma(\Delta) \in Tr(S)) \cdot \mathbb{P}(\sigma)
= (1+\psi_1(s)) \sum_{\Delta\in T(H)} \frac{1}{s} = \frac{|T(H)|}{s} (1+\psi_1(s)).$$

\end{proof}

\begin{corollary}\label{Cor:s2-eps}

Let $ H $ be a graph on $ s $ vertices. If the inequality $|T(H)| \leq s^{3-\delta}$ holds for some $\delta > 0,$ then there exists a permutation
$\sigma$ of the set $V(H)$ such that $F(\sigma,H) \leq s^{2-\delta}(1+\psi_1(s))$.

\end{corollary}

\begin{proof} [Proof of Theorem \ref{T:K3-K3}]

Let $ G \sim G(n,1/2) $. 

Let $ \delta $ from Corollary \ref{Cor:s2-eps}
be equal to $ \eps$ from Section \ref{s3.1}. Set $ z = s^{2-\varepsilon} (1+\psi_1(s)) $.

For any $s$-subset $S$ of the set $V(G)$ we have

$$\mathbb{P}(A_S)= \mathbb{P}\left(|T(G[S])| \leq s^{3-\varepsilon}\right) \leq $$

(using Corollary \ref{Cor:s2-eps})

$$ \leq   \mathbb{P}\left(\bigcup_{\sigma} \left(F\left(\sigma,G[S]\right) \leq z\right)\right)
\leq \sum_{\sigma} \sum_{i=0}^{z} \mathbb{P}\left(F\left(\sigma,G[S]\right)=i\right) = s! \cdot \sum_{i=0}^{z} \mathbb{P}\left(F\left(\sigma,G[S]\right)=i\right),$$
where $ \sigma $ is an
arbitrary permutation.

Let us bound the sum. Put $ a = \frac{s^2}{6}(1+\psi(s)). $ Taking into account that $ |Tr(S)| = a $ (we also assume that $ s $ is such that $ a/2 > z $) we obtain:

$$\sum_{i=0}^{z} \mathbb{P}\left(F\left(\sigma,G[S]\right)=i\right) =
\sum_{i=0}^{z} {a\choose i}\cdot \left(\frac{1}{8}\right)^{i}\cdot \left(\frac{7}{8}\right)^{a-i}
\leq (z+1) a^z \left(\frac{7}{8}\right)^{a} = $$

$$ = 2^{o(s^{2})} \left(\frac{7}{8}\right)^{\frac{s^2}{6}(1+o(1))} = \left(\frac{7}{8}\right)^{{\ss} (1+o(1))}.$$

By symmetry, $\mathbb{P}(A'_S)$ can be bounded analogously.
\end{proof}

Next we describe how to improve the obtained bound.
Take a graph $H = (S,E)$ of order $s$. Instead of $Tr(S)$ we consider a maximum system of pairwise disjoint graphs isomorphic to $K_k$ on the set of vertices $S=\{1,\dots,s\}$.
Let $Sys(S,k)$  denote one such system.
For a fixed $ k $ and for $ s \to \infty $ R\"odl's theorem implies that $|Sys(S,k)| \sim \frac{s^2}{k(k-1)},$ or, equivalently,
$|Sys(S,k)| = \frac{s^2}{k(k-1)}(1+\xi_k(s))$.

Let $\sigma$ be a permutation of the set $V(H)$. Let $F_k(\sigma,H) $ denote the number of such triangles from the set
$T(\sigma(H))$  that are subgraphs of one of the complete subgraphs of size $k$ from $Sys(S,k)$.
Below we indicate the changes in the proof of Theorem \ref{T:K3-K3}. Assume $k \geq 4$.

Let us generalize Claim \ref{C:Mat(F)}. Before the claim we defined $ \psi_1 $. Similarly to how we defined $\psi_1$ based on $\psi$ we define $ \xi^1_k $ based on $ \xi_k $.

\begin{claim}\label{C:Mat(Fk)} Fix a natural $k\ge 4$. For every graph $ H $ with $ s $ vertices we have:

$$\mathbb{E}(F_k(\sigma,H))  = \dfrac{(k-2)|T(H)|}{s} (1+\xi_k^1(s)).$$

\end{claim}

\begin{proof}

The proof is similar to the proof of Claim \ref{C:Mat(F)}. We point out several differences in calculations.

Let $\Delta \in T(H)$. For every $k$-clique $K_k \in Sys(S,k)$ the number of permutations $\sigma$ such that $K_k$ contains $\sigma(\Delta)$ as a subgraph, equals
$(s-3)!k(k-1)(k-2)$. Thus, the number of permutations $\sigma$ such that $\sigma(\Delta) \in Sys(S,k)$ equals $(s-3)!k(k-1)(k-2) \cdot |Sys(S,k)|$.

This implies

$$
\sum_{\sigma} \mathbb{I}(\sigma(\Delta) \in Sys(S,k)) \cdot \mathbb{P}(\sigma) = \frac{1}{s!} (s-3)!k(k-1)(k-2) \cdot \frac{s^2}{k(k-1)} (1+\xi_k(s)) =
\frac {k-2} s  (1+\xi_k^1(s)),
$$

$$ \mathbb{E} (F_k(\sigma,H)) = \sum_{\Delta\in T(H)} \sum_{\sigma} \mathbb{I}(\sigma(\Delta) \in Sys(S,k)) \cdot \mathbb{P}(\sigma)
= \frac{(k-2)|T(H)|}{s} (1+\xi_k^1(s)).$$

\end{proof}

\begin{corollary}\label{Cor5}

Fix a natural $k$ greater than 4 and positive $\delta$. Let $ H $ by a graph on $ s $ vertices. If
$|T(H)| \le s^{3-\delta},$ then there exists a permutation $\sigma$ of the set $V(H)$ such that $F_k(\sigma,H) \leq (k-2)s^{2-\delta}(1+\xi_k^1(s))$.

\end{corollary}

In the case $k=4$ this corollary gives the following theorem.

\begin{theorem}\label{T:K3-K4}
$$
\mathbb{P}(A_S)\le\mathcal{P}, \, \mathbb{P}(A_S')\le\mathcal{P}, \hbox{where} \,\, \mathcal{P} = s! \cdot \left(\frac{41}{64}\right)^{{\frac{s^{2}}{12}}(1+o(1))}.
$$
\end{theorem}

\begin{proof}

The  proof is analogous to the proof of Theorem \ref{T:K3-K3}. While in that proof we used Corollary \ref{Cor:s2-eps}, here we apply Corollary \ref{Cor5}.
 We use the same notation as in the proof of Theorem \ref{T:K3-K3}. That is, let $ \eps $ be the one appeared in Section \ref{s3.1}.
Put $ \delta $ from Corollary \ref{Cor5} to be equal to $ \eps $. We have the following equality: $ z = 2 s^{2-\eps}(1+\xi_4'(s)) $.

We already know that $|Sys(S,4)| = \frac{s^2}{12}(1+\xi_4(s))$. Hence $ a = \frac{s^2}{12}(1+\xi_4(s))$. In fact, to complete the proof it remains to prove that

$$
\sum_{i=0}^{z} \mathbb{P}\left(F_k\left(\sigma,G[S]\right)=i\right) \le \left(\frac{41}{64}\right)^{{\frac{s^{2}}{12}}(1+o(1))}.
$$
The event $ \left\{F_k\left(\sigma,G[S]\right)=i\right\} $ implies the following event: at most $ i $ cliques from $ Sys(S,4) $ contain
at least one triangle from the graph $\sigma(G[S])$. At the same time the probability of the event that $G(4,1/2)$ does not contain any triangles is
$\frac{41}{64}$. Therefore, for large $ s $ we have:

$$
\sum_{i=0}^{z} \mathbb{P}\left(F_k\left(\sigma,G[S]\right)=i\right) \leq
\sum_{i=0}^{z} \sum_{j=0}^i {a\choose j} \cdot \left(\frac{23}{64}\right)^{j} \cdot
\left(\frac{41}{64}\right)^{a-j} \leq \sum_{i=0}^{z} (i+1)
a^i \left(\frac{41}{64}\right)^{a} \le
$$
$$
\leq (z+1)^2 a^z \left(\frac{41}{64}\right)^{\frac{s^2}{12} (1+o(1))} = \left(\frac{41}{64}\right)^{\frac{s^2}{12} (1+o(1))},
$$
which completes the proof.

\end{proof}

Analogously to Corollary \ref{Cor:K3-K3} we obtain

\begin{corollary}\label{Cor6} For $d\in \{4, 5\}$ the following lower bound holds:

$$R_{{\it D}}(s,s,d) \geq \left({\frac{64}{41}} \right)^{{\frac {s}{12}} (1+o(1))} \approx 2^{0.053537s}.$$

\end{corollary}

We use $\mathcal{P}(k,l)$ to denote the probability that the random graph $G(k,{1}/{2})$ does not have subgraphs isomorphic to $K_l$.
One can easily generalize the above described method (Corollaries \ref{Cor:K3-K3} and \ref{Cor6}). Thus, for $ d\in \{4, 5\} $ we obtain the following bound:
$$
R_{{\it D}}(s,s,d) \geq \left(\frac{1}{\mathcal{P}(k,3)}\right)^{{\frac {s}{k(k-1)}} (1+o(1))}.
$$
Let us note that in this bound the value $ o(1) $ depends both on $ k $ and $ s $, so we apply this bound for fixed $k$ and for $s$ that tends to infinity.

It is known that (see a more general claim in the next section)

$$
\mathcal{P}(k,3) = \frac{2^{k^2/4 + f_1(k)}}{2^{k\choose 2}} = 2^{-k^2/4 + f_2(k)}, ~~~ f_1(k) = o(k^2), ~~~ f_2(k) = o(k^2).
$$
Hence
$$
\left(\frac{1}{\mathcal{P}(k,3)}\right)^{\frac {s}{k(k-1)}} = 2^{(1/4-f_3(k))s}, ~~~ \lim_{k \to \infty} f_3(k) = 0.
$$
First we fix large $ k $, next choose a sufficiently large $ s $. Finally we get:
$$
R_{{\it D}}(s,s,d) \geq \left(\frac{1}{\mathcal{P}(k,3)}\right)^{\frac{s}{k(k-1)} (1+o(1))} = \left(2^{(1/4-f_3(k))s}\right)^{1+o(1)} >
2^{(1/4-\gamma)s}.
$$
This concludes the proof of Theorem \ref{T:main} for  $ d\in \{4, 5\} $.

\bigskip

\subsection{Cases $d \ge 6$}\label{s3.3}

We generalize the method, described in the previous section, to the case of arbitrary $ d $.
While there we considered triangles, now we deal with $ l $-cliques, where $ l = [d/2]+1 $. Instead of
$ F_k(\sigma,H) $ we consider random variables $ F^l_k(\sigma,H) $, where $ F^l_k(\sigma,H) $ is the number of such $ l $-cliques in $ \sigma(H) $
that are contained as a subgraph in one of the $k$-cliques from $Sys(S,k)$.

Let us give the analogue of Claim \ref{C:Mat(Fk)}.

\begin{claim} Fix natural $k, l$, $l\le k$. For every graph $ H $ with $ s $ vertices we have:
$$
\mathbb{E} \left(F^l_k(\sigma,H)\right) = \frac{(k-2) \cdot \ldots \cdot (k-l+1)cl(H,l)}{s^{l-2}}\left(1+\zeta^l_k(s)\right).
$$
\end{claim}

We omit here the proof of the claim, the corollary and futher calculations.

It is clear that finally one gets
$$
R_{{\it D}}(s,s,d) \geq \left(\frac{1}{\mathcal{P}(k,l)}\right)^{{\frac {s}{k(k-1)}} (1+o(1))},
$$
where for fixed $ d $ the value $ o(1) $ depends only on $ k $ and $ s $.

It was shown in the paper \cite{ER} that, for fixed natural $l$ greater than $ 3,$ the number of graphs with $ k $ vertices and without $ l $-cliques is
$$
2^{\frac{k^2}{2}\left(1-\frac{1}{l-1}\right)+f(k,l)},
$$
where the value of $ f(k,l) $ is $ o(k^2) $. Further calculations reproduce those from the end of the previous section.

\section{Proof of Proposition \ref{upper}}

Note that every  $ [d/2]$-partite graph can be realized as a distance graph in $ {\mathbb R}^d $. Indeed, consider circles $C_i$, $i=1,\ldots, [d/2]:$ 

$$C_i=\{(0,\ldots, 0, x_{2i-1},x_{2i},0,\ldots,0)\in \R^d:x_{2i-1}^2+x_{2i}^2=1/2\}.$$
 
 Embed the $i$th part of the multipartite graph into $C_i$. By Pithagoras' theorem, the distance between any two points from $C_i,C_j,$ for distinct  $i$ and $j,$ equals 1.

 So, to prove the proposition it is enough to show that for every graph with
$$
m = 2 \left[\frac{d}{2}\right] R\left(\left \lceil \frac{s}{[d/2]} \right \rceil, \left \lceil \frac{s}{[d/2]} \right \rceil\right)
$$
vertices the following holds: either the graph or its complement has $ [d/2] $ independent sets with total cardinality at least $ s $.
Take a graph $ G=(V,E) $ on $m$ vertices. Split its vertex set into $ t=2[d/2] $ parts so that each part has cardinality
$$
\frac{m}{t} = R\left(\left \lceil \frac{s}{[d/2]} \right \rceil, \left \lceil \frac{s}{[d/2]} \right \rceil\right).
$$
Let $ V_1, \dots, V_t $ denote these parts. Put $ G_1 = G[V_1] $, $ \dots, $ $ G_t = G[V_t] $. By the definition of the classical Ramsey number
for every $ i \in \{1, \dots, t\} $ either $ G_i $ or $ \bar{G}_i $ has an independent set with cardinality
$y= \left \lceil \frac{s}{[d/2]} \right \rceil $. Assume that (without loss of generality) there are at least $ [d/2] = t/2 $ indexes $ i $ such that
$ G_i $ has an independent set of size $y$. Take a union of the collection of $ G_i $ over $t/2$ such indexes $i$. The union is a subgraph in $ G,$ which is realizable as distance graph in $\R^d$ and  and already has  at least $yt/2$ vertices, and $ yt/2\ge s $. This concludes the proof.

\end{document}